\documentclass[12pt,reqno,a4paper]{amsart}
\usepackage{amssymb}
\usepackage{mathrsfs}
\usepackage{amsfonts}
\usepackage{latexsym,amsmath,amsfonts,amssymb,amsthm}
\theoremstyle{plain}
\newtheorem{Thm}{Theorem}
\newtheorem{Lem}{Lemma}

\theoremstyle{definition}

\theoremstyle{remark}

\def\Z{\mathbb Z}
\def\N{\mathbb N}

\def\C{\mathbb C}

\def\F{\mathcal {F}}

\def\fS{\mathfrak S}
\def\T{\mathbb T}
\def\R{\mathbb R}
\def\P{\mathcal P}

\def\M{\mathfrak M}
\def\m{\mathfrak m}

\def\a{{\mathfrak a}}
\def\b{{\mathfrak b}}

\def\RR{\mathcal R}
\def\B{\mathcal B}

\def\1{{\bf 1}}
\def\pmod #1{\ ({\rm mod}\ #1)}

\begin{document}
\title{Chen's primes and ternary Goldbach problem}
\author{Hongze Li}
\address{
Department of Mathematics, Shanghai Jiaotong University, Shanghai
200240, People's Republic of China} \email{lihz@sjtu.edu.cn}
\author{Hao Pan}
\address{Department of Mathematics, Nanjing University,
Nanjing 210093, People's Republic of
China}\email{haopan79@yahoo.com.cn} \thanks{This work was supported by the
National Natural Science Foundation of China (Grant No. 10771135).}
\maketitle
\begin{abstract}
We prove that there exists a $k_0>0$ such that every sufficiently
large odd integer $n$ with $3\mid n$ can be represented as
$p_1+p_2+p_3$, where $p_1,p_2$ are Chen's primes and $p_3$ is a
prime with $p_3+2$ has at most $k_0$ prime factors.
\end{abstract}

\section{Introduction}
\setcounter{equation}{0} \setcounter{Thm}{0} \setcounter{Lem}{0}
\setcounter{Cor}{0}

Let $\P$ denote the set of all primes. Define
$$
\P_k=\{n:\, n\in \mathbb{N} \text{ and } n\text{ has at most
}k\text{ prime divisors}\}
$$
and
$$
\P_k^{(2)}=\{p\in\P:\, p+2\in\P_k\}.
$$
The well-known twin primes conjecture asserts that $\P_1^{(2)}$
has infinitely many elements. Nowadays, the best result on the
twin primes conjecture belongs to Chen \cite{Chen73}, who proved
that $\P_2^{(2)}$ has infinitely many elements. In fact, Chen
proved that for sufficiently large $x$,
$$
|\{p\in\P_2^{(2)}:\, p\leq x,\
(p+2,P(x^{1/10}))=1\}|\gg\frac{x}{(\log x)^2},
$$
where
$$
P(z)=\prod_{p<z}p.
$$
In Iwaniec's unpublished notes \cite{Iwaniec96}, the exponent
$1/10$ can be improved to $3/11$. In \cite{GreenTao08}, Green and
Tao say a prime $p$ is Chen's prime if $p\in\P_2^{(2)}$.

On the other hand, in 1937 Vinogradov \cite{Vinogradov37} solved
the ternary Goldbach problem and showed that every sufficiently
large odd integer can be represented as the sum of three primes.
Two years later, using Vinogradov's method, van der Corput
\cite{Corput39} proved that the primes contain infinitely many
non-trivial three term arithmetic progressions (3AP). In 1999,
with the help of the vector sieve method, Tolev \cite{Tolev99}
proved that there exist infinitely many non-trivial 3APs
$\{p_1,p_2,p_3\}$ of primes, satisfying $p_1\in\P_4^{(2)}$,
$p_2\in\P_5^{(2)}$ and $p_3\in\P_{11}^{(2)}$.

However, in \cite{GreenTao08}, with the help of the Szemer\'edi
theorem, their transference principle and a result of Goldston and
Y\i ld\i r\i m, Green and Tao proved that the primes contain
arbitrarily long non-trivial arithmetic progressions. Certainly this
is a remarkable breakthrough in additive number theory. Furthermore,
Green and Tao also claimed that using their method, one can prove
that Chen's primes contain arbitrarily long non-trivial arithmetic
progressions. And for the 3APs of Chen's primes, they proposed a
detail proof in \cite{GreenTao06}.

Let us return to the ternary Goldbach problems. In \cite{Peneva00},
using Tolev's method, Peneva proved that every sufficiently large
odd integer $n$ with $3\mid n$ can be represented as $n=p_1+p_2+p_3$
with $p_i\in\P_{k_i}^{(2)}$, $k_1=k_2=5$ and $k_3=8$. Subsequently,
Tolev \cite{Tolev00} improved Peneva's result to $k_1=2$, $k_2=5$
and $k_3=7$. Recently, Meng \cite{Meng07} proved that every
sufficiently large odd integer $n$ with $3\nmid n-1$ can be
represented as $n=p_1+p_2+p_3$, where $p_1$ is Chen's prime,
$p_2\in\P_3^{(2)}$ and $p_3\in\P$ (not of special type!).

Of course, we wish to prove that every sufficiently large odd
integer $n$ with $3\mid n$ can be represented as the sum of three
Chen's primes. Unfortunately, as we shall see later, it seems not
easy. The key of Green and Tao's proof in \cite{GreenTao06} is to
transfer Chen's primes to a subset with positive density of
$\Z_N=\Z/N\Z$ (where $N$ is a large prime). But this density is too
small. However, in the present paper, we shall prove the following
 result:
\begin{Thm}
\label{chengoldbach} There exists a positive integer $k_0$ such that
every sufficiently large odd integer $n$ with $3\mid n$ can be
represented as $p_1+p_2+p_3$, where $p_1,p_2$ are Chen's primes and
$p_3\in\P_{k_0}^{(2)}$.
\end{Thm}

In Sections 2 and 3, we shall estimate some exponential sums
involving the primes of special type. The proof of Theorem
\ref{chengoldbach} will be given in Section 4.

\section{The Minor Arcs}
\setcounter{equation}{0} \setcounter{Thm}{0} \setcounter{Lem}{0}
\setcounter{Cor}{0}

Let $k_0\geq 8$ be a fixed integer and $B=6^9$. Suppose that $n$ is
a sufficiently large integer, $W>0$ is an even integer with $W\leq
(\log n)^B$, and $1\leq b\leq W$ satisfies $(b(b+2),W)=1$. Let $\T$
denote the torus $\R/\Z$. For $1\leq q\leq(\log n)^B$, define
$$
\M_{a,q}=\{\alpha\in\T:\,|\alpha q-a|\leq(\log n)^B/n\}.
$$
Let
$$
\M=\bigcup_{\substack{1\leq a\leq q\leq (\log
n)^B\\(a,q)=1}}\M_{a,q}
$$
and $\m=\T\setminus\M$.

Let $D=n^{0.32}$ and $z_0=n^{1/k_0}$. For square-free number
$d=p_1p_2\cdots p_k$ with primes $p_1>p_2>\cdots>p_k$, define
Rosser's weights with the order $D$ by
$$
\lambda_{D}^{+}(d)=\begin{cases} (-1)^k&\text{if }p_1\cdots
p_{2l}p_{2l+1}^3<D\text{ for all }0\leq l\leq (k-1)/2,\\
0&\text{otherwise},
\end{cases}
$$
and
$$
\lambda_{D}^{-}(d)=\begin{cases} (-1)^k&\text{if }p_1\cdots
p_{2l-1}p_{2l}^3<D\text{ for all }1\leq l\leq k/2,\\
0&\text{otherwise}.
\end{cases}
$$
It is easy to see that $\lambda_D^{\pm}(d)=0$ if $d\geq D$. Let $F(s)$ and $f(s)$ denote the functions of linear sieve.
The following lemma is a fundamental result in sieve method.
\begin{Lem}[Iwaniec \cite{Iwaniec80a, Iwaniec80b}]
\label{iwaniec}
Suppose that $\P_*$ is any set of primes and $\omega$ is a multiplicative function satisfying:
$$
0<\omega(p)<p\text{ for }p\in\P_*,\ \omega(p)=0\text{ for }p\not\in\P_*,
$$
and
$$
\prod_{z_1\leq p<z_2}\bigg(1-\frac{\omega(p)}{p}\bigg)^{-1}\leq\frac{\log z_2}{\log z_1}\bigg(1+\frac{L}{\log z_1}\bigg)
$$
for a constant $L>0$ and for all $2\leq z_1\leq z_2$. Then we have
\begin{align}
\label{rosserF}
\prod_{p<z}\bigg(1-\frac{\omega(p)}{p}\bigg)\leq&\sum_{d\mid P_*(z)}\lambda_D^+(d)\frac{\omega(d)}{d}\notag\\
\leq&\prod_{p<z}\bigg(1-\frac{\omega(p)}{p}\bigg)(F(s)+O(e^{\sqrt{L}-s}(\log
D)^{-1/3})),
\end{align}
provided that $2\leq z\leq D$, where $s=\log D/\log z$ and
$$
P_*(z)=\prod_{p\in\P_*\cap[1,z)}p.
$$
Similarly,
\begin{align}
\label{rosserf}
\prod_{p<z}\bigg(1-\frac{\omega(p)}{p}\bigg)\geq&\sum_{d\mid P_*(z)}\lambda_D^-(d)\frac{\omega(d)}{d}\notag\\
\geq&\prod_{p<
z}\bigg(1-\frac{\omega(p)}{p}\bigg)(f(s)+O(e^{\sqrt{L}-s}(\log
D)^{-1/3})),
\end{align}
provided that $2\leq z\leq \sqrt{D}$. Furthermore, for any
square-free integer $q$,
\begin{align}
\label{rosserpm}
\sum_{d\mid q}\lambda_D^-(d)\leq\sum_{d\mid q}\mu(d)\leq\sum_{d\mid q}\lambda_D^+(d).
\end{align}
\end{Lem}
Define
$$
S_{n,z_0}(\alpha)=\sum_{\substack{p\leq n\\ p\equiv b\pmod{W}\\
(p+2,P(z_0))=1}}e(\alpha(p-b)/W)\log p.
$$
Clearly
$$
S_{n,z_0}(\alpha)=\sum_{\substack{p\leq n\\ p\equiv
b\pmod{W}}}e(\alpha(p-b)/W)\log p\sum_{d\mid (p+2,P(z_0))}\mu(d).
$$
Let
$$
S_{n,z_0}^{\pm}(\alpha)=\sum_{\substack{p\leq n\\ p\equiv
b\pmod{W}}}e(\alpha(p-b)/W)\log p\sum_{d\mid
(p+2,P(z_0))}\lambda_D^\pm(d).
$$
\begin{Lem}
\label{spm} For any $\alpha\in\T$, we have
$|S_{n,z_0}^{+}(\alpha)-S_{n,z_0}(\alpha)|\leq
S_{n,z_0}^{+}(0)-S_{n,z_0}(0)$ and
$|S_{n,z_0}(\alpha)-S_{n,z_0}^{-}(\alpha)|\leq
S_{n,z_0}(0)-S_{n,z_0}^-(0)$.
\end{Lem}
\begin{proof}
By (\ref{rosserpm}),
\begin{align*}
|S_{n,z_0}^{+}(\alpha)-S_{n,z_0}(\alpha)|\leq&\sum_{\substack{p\leq n\\
p\equiv b\pmod{W}}}\log p\bigg|\sum_{d\mid
(p+2,P(z_0))}\lambda^+(d)-\sum_{d\mid
(p+2,P(z_0))}\mu(d)\bigg|\\
=&S_{n,z_0}^{+}(0)-S_{n,z_0}(0).
\end{align*}
The proof of the second inequality is similar.
\end{proof}

Let $\tau$ denote the divisor function.
It is well-known
$$
\sum_{d\leq X}\tau(d)^A\ll_AX(\log X)^{2^A-1}
$$
and
$$
\sum_{d\leq X}\frac{\tau(d)^A}{d}\ll_A(\log X)^{2^A}.
$$
\begin{Lem}
Suppose that $X,X',Y,Y',Z,Z'>0$ satisfy $X\leq X'\leq 2X$, $Y\leq
Y'\leq 2Y$ and $XY\leq Z\leq Z'\leq X'Y'$. For any $d\geq 1$, let
$u_d$, $v_d$, $w_d$ be complex numbers with $|u_d|, |v_d|,
|w_d|\leq\tau(d)^{A}(\log(XY))^{A}$. Suppose that $1\leq a\leq q$
with $(a,q)=1$, and $\alpha\in\T$ with $|\alpha q-a|\leq 1/q$.
Then
\begin{align}
\label{expsum1}
&\sum_{X\leq x\leq X'}u_x\bigg|\sum_{\substack{Y\leq y\leq Y'\\ Z\leq xy\leq Z'\\ xy\equiv b\pmod{W}}}v_ye(\alpha (xy-b)/W)\sum_{\substack{d\mid xy+2\\
d\leq D}}w_d\bigg|\notag\\
\ll&_A
XY(\log(DXYq))^{2^{2A+2}}\tau^{2A+2}(W)\bigg(\frac{1}{q}+\frac{D^2W}{X}+\frac{qW}{XY}\bigg)^{1/4}\notag\\
&+XY^{1/2}(\log(DXY))^{2^{2A+2}}\tau^{2A+1}(W)
\end{align}
provided that $X\geq D^2W$. Furthermore, suppose that $0\leq v_{y_1}\leq v_{y_2}\leq (\log(XY))^A$ for any $y_1,y_2$ with $y_1\leq y_2$. Then
\begin{align}
\label{expsum2}
&\sum_{X\leq x\leq X'}u_x\bigg|\sum_{\substack{Y\leq y\leq Y'\\ Z\leq xy\leq Z'\\ xy\equiv b\pmod{W}}}v_ye(\alpha (xy-b)/W)\sum_{\substack{d\mid xy+2\\
d\leq D}}w_d\bigg|\notag\\
\ll&_A
XY(\log(DXYq))^{6^{A+2}}\bigg(\frac{1}{q}+\frac{DW}{Y}+\frac{qW}{XY}\bigg)^{1/4}
\end{align}
provided that $Y\geq DW$.
\end{Lem}
\begin{proof} By the Cauchy-Schwarz inequality,
\begin{align*}
&\bigg(\sum_{X\leq x\leq X'}u_x\bigg|\sum_{\substack{Y\leq y\leq Y'\\ Z\leq xy\leq Z'\\ xy\equiv b\pmod{W}}}v_ye(\alpha (xy-b)/W)\sum_{\substack{d\mid xy+2\\
d\leq D}}w_d\bigg|
\bigg)^2\\
\leq&\bigg(\sum_{X\leq x\leq X'}|u_x|^2\bigg)
\bigg(\sum_{\substack{X\leq x\leq X'\\
d_1,d_2\leq
D}}w_{d_1}\overline{w_{d_2}}\sum_{\substack{Y\leq
y_1,y_2\leq Y'\\ Z\leq xy_1,xy_2\leq Z'\\ xy_1,xy_2\equiv b\pmod{W}\\ xy_i\equiv
-2\pmod{d_i}\text{ for }i=1,2}}v_{y_1}\overline{v_{y_2}}e(\alpha x(y_1-y_2)/W)\bigg)\\
\ll&_AX(\log X)^{6^A}\sum_{\substack{1\leq b'\leq W,\ (b',W)=1\\
Y\leq
y_1,y_2\leq Y'\\ y_1\equiv y_2\equiv b'\pmod{W}\\
d_1,d_2\leq
D}}v_{y_1}\overline{v_{y_2}}w_{d_1}\overline{w_{d_2}}\sum_{\substack{X\leq x\leq X'\\
Z\leq xy_1,xy_2\leq Z'\\
xb'\equiv b\pmod{W}\\ xy_i\equiv -2\pmod{d_i}\text{ for
}i=1,2}}e(\alpha x(y_1-y_2)/W).\\
\end{align*}
We have
\begin{align*}
&\sum_{\substack{X\leq x\leq X'\\
Z\leq xy_1,xy_2\leq Z'\\
xb'\equiv b\pmod{W}\\ xy_i\equiv -2\pmod{d_i}\text{ for
}i=1,2}}e(\alpha x(y_1-y_2)/W)\\
=&\sum_{\substack{\max\{X, Z/y_1, Z/y_2\}\leq x\leq \min\{X', Z'/y_1,Z'/y_2\}\\
xb'\equiv b\pmod{W}\\ xy_i\equiv -2\pmod{d_i}\text{ for
}i=1,2}}e(\alpha x(y_1-y_2)/W)\\
\ll&\min\bigg\{\frac{X}{[d_1,d_2,W]},\frac{1}{\|\alpha
[d_1,d_2,W](y_1-y_2)/W\|}\bigg\},
\end{align*}
where $\|\theta\|=\min\{|\theta-t|:\, t\in\Z\}$. Hence for each $1\leq b'\leq W$ with $(b',W)=1$, we have
\begin{align*}
&\sum_{\substack{Y\leq y_1,y_2\leq Y'\\ y_1\equiv y_2\equiv
b'\pmod{W}\\
d_1,d_2\leq D}
}v_{y_1}\overline{v_{y_2}}w_{d_1}\overline{w_{d_2}}\bigg|\sum_{\substack{X\leq x\leq X'\\
Z\leq xy_1,xy_2\leq Z'\\ xb'\equiv b\pmod{W}\\ xy_i\equiv -2\pmod{d_i}\text{ for
}i=1,2}}e(\alpha x(y_1-y_2)/W)\bigg|\\
\ll&\sum_{\substack{1\leq y_1,y_2\leq Y'\\ y_1\equiv y_2\equiv
b'\pmod{W}\\
d_1,d_2\leq D}}|v_{y_1}||v_{y_2}||w_{d_1}||w_{d_2}|\min\bigg\{\frac{X}{[d_1,d_2,W]},\frac{1}{\|\alpha
[d_1,d_2,W](y_1-y_2)/W\|}\bigg\}\\
\ll&Y(\log(XY))^{6^A}\bigg(\sum_{\substack{h\leq
D^2Y}}\tau^{4A+3}(hW)\min\bigg\{\frac{XY}{hW},\frac{1}{\|\alpha
h\|}\bigg\}+\frac{X}{W}\sum_{\substack{h\leq
D^2}}\frac{\tau^{4A+2}(hW)}{h}\bigg)\\
&(\text{where }h=[d_1,d_2,W](y_1-y_2)/W\text{ if }y_1>y_2, \text{ and }h=[d_1,d_2,W]/W)\text{ if }y_1=y_2\\
\ll&Y(\log(DXY))^{6^A+1}\tau^{4A+3}(W)\max_{H\leq
D^2Y}\sum_{\substack{H/2\leq h\leq
H}}\tau^{4A+3}(h)\min\bigg\{\frac{XY}{HW},\frac{1}{\|\alpha
h\|}\bigg\}\\
&+\frac{XY\tau^{4A+2}(W)}{W}(\log(DXY))^{2^{4A+3}}.
\end{align*}
Applying Lemma 2.2 of \cite{Vaughan97}, for any $H\leq D^2Y$,
\begin{align*}
&\sum_{\substack{H/2\leq h\leq
H}}\tau^{4A+3}(h)\min\bigg\{\frac{XY}{HW},\frac{1}{\|\alpha
h\|}\bigg\}\\
\leq&\bigg(\sum_{\substack{h\leq
H}}\tau^{8A+6}(h)\bigg)^{1/2}\bigg(\frac{XY}{HW}\sum_{\substack{h\leq
H}}\min\bigg\{\frac{XY}{hW},\frac{1}{\|\alpha
h\|}\bigg\}\bigg)^{1/2}\\
\ll&_A\bigg(\frac{X^2Y^2(\log(DYq))^{2^{8A+6}}}{W^2}\bigg(\frac{1}{q}+\frac{D^2W}{X}+\frac{qW}{XY}\bigg)\bigg)^{1/2}.
\end{align*}
This concludes the proof of (\ref{expsum1}).

Let us turn to (\ref{expsum2}). Clearly
\begin{align*}
&\bigg(\sum_{X\leq x\leq X'}u_x\bigg|\sum_{\substack{Y\leq y\leq Y'\\ Z\leq xy\leq Z'\\ xy\equiv b\pmod{W}}}v_ye(\alpha (xy-b)/W)\sum_{\substack{d\mid xy+2\\
d\leq D}}w_d\bigg|\bigg)^2\\
\leq&\bigg(\sum_{X\leq x\leq X'}|u_x|^2\bigg)\bigg(\sum_{X\leq x\leq X'}\bigg|\sum_{\substack{Y\leq y\leq Y'\\ Z\leq xy\leq Z'\\ xy\equiv b\pmod{W}}}v_ye(\alpha (xy-b)/W)\sum_{\substack{d\mid xy+2\\
d\leq D}}w_d\bigg|^2\bigg)\\
\ll&X(\log X)^{6^A}\sum_{\substack{
1\leq b'\leq W,\ (b',W)=1\\ X\leq x\leq X'\\
x\equiv b'\pmod{W}\\
d_1,d_2\leq D}}|w_{d_1}||w_{d_2}|
\prod_{i=1}^2\bigg|\sum_{\substack{Y\leq
y_i\leq Y'\\ Z\leq xy_i\leq Z'\\ y_ib'\equiv b\pmod{W}\\
xy_i\equiv -2\pmod{d_i}}}v_{y_i}e(\alpha xy_i/W)\bigg|.
\end{align*}
By the partial summation,
\begin{align*}
&\sum_{\substack{Y\leq
y_i\leq Y'\\ Z\leq xy_i\leq Z'\\ y_ib'\equiv b\pmod{W}\\
xy_i\equiv -2\pmod{d_i}}}v_{y_i}e(\alpha xy_i/W)\\
=&v_{Y'}\sum_{\substack{1\leq
y_i\leq Y'\\ Z\leq xy_i\leq Z'\\ y_ib'\equiv b\pmod{W}\\
xy_i\equiv -2\pmod{d_i}}}e(\alpha xy_i/W)-
v_{Y}\sum_{\substack{1\leq
y_i\leq Y-1\\ Z\leq xy_i\leq Z'\\ y_ib'\equiv b\pmod{W}\\
xy_i\equiv -2\pmod{d_i}}}e(\alpha xy_i/W)\\
&-\sum_{Y\leq Y''\leq Y'-1}(v_{Y''+1}-v_{Y''})\sum_{\substack{1\leq
y_i\leq Y''\\ Z\leq xy_i\leq Z'\\ y_ib'\equiv b\pmod{W}\\
xy_i\equiv -2\pmod{d_i}}}e(\alpha xy_i/W)\\
\ll&\bigg(v_{Y'}+v_{Y}
+\sum_{Y\leq Y''\leq Y'-1}(v_{Y''+1}-v_{Y''})\bigg)
\min\bigg\{\frac{Y}{[d_i,W]},\frac{1}{\|\alpha
[d_i,W]x/W\|}\bigg\}\\
\ll&_A(\log(XY))^A
\min\bigg\{\frac{Y}{[d_i,W]},\frac{1}{\|\alpha
[d_i,W]x/W\|}\bigg\}.\\
\end{align*}
For any $1\leq b'\leq W$ with $(b',W)=1$,
\begin{align*}
&\sum_{\substack{
X\leq x\leq X'\\
x\equiv b'\pmod{W}\\
d_1,d_2\leq D}}|w_{d_1}||w_{d_2}|
\prod_{i=1}^2\min\bigg\{\frac{Y}{[d_i,W]},\frac{1}{\|\alpha
[d_i,W]x/W\|}\bigg\}\\
\leq&(\log(XY))^{2A}\sum_{\substack{X\leq x\leq X'\\
x\equiv b'\pmod{W}}}\bigg(\sum_{\substack{
d\leq D}}\tau(d)^A\min\bigg\{\frac{Y}{[d,W]},\frac{1}{\|\alpha
[d,W]x/W\|}\bigg\}\bigg)^2\\
\leq&(\log(XY))^{2A}\sum_{\substack{X\leq x\leq X'\\
x\equiv b'\pmod{W}}}\bigg(\sum_{\substack{
h\leq D}}\tau(hW)^{A+1}\min\bigg\{\frac{Y}{hW},\frac{1}{\|\alpha
hx\|}\bigg\}\bigg)^2\\
&(\text{where }h=[d,W]/W)\\
\leq&(\log(XY))^{2A}\sum_{\substack{X\leq x\leq X'\\
x\equiv b'\pmod{W}}}\bigg(\sum_{2^k\leq D}\sum_{\substack{ 2^k\leq
h\leq
2^{k+1}}}\tau(hW)^{A+1}\min\bigg\{\frac{Y}{hW},\frac{1}{\|\alpha
hx\|}\bigg\}\bigg)^2\\
\ll&(\log(DXY))^{2A+2}\max_{H\leq D}\sum_{\substack{X\leq x\leq X'\\
x\equiv b'\pmod{W}}}\bigg(\sum_{\substack{H/2\leq h\leq H}}\tau(hW)^{A+1}\min\bigg\{\frac{Y}{HW},\frac{1}{\|\alpha
hx\|}\bigg\}\bigg)^2.
\end{align*}
And for any $H\leq D$,
\begin{align*}
&\sum_{\substack{X\leq x\leq X'\\
x\equiv b'\pmod{W}}}\bigg(\sum_{\substack{H/2\leq h\leq H}}\tau(hW)^{A+1}\min\bigg\{\frac{Y}{HW},\frac{1}{\|\alpha
hx\|}\bigg\}\bigg)^2\\
\leq&\sum_{\substack{X\leq x\leq X'\\
x\equiv b'\pmod{W}}}\bigg(\tau(W)^{2A+2}\sum_{\substack{h\leq H}}\tau(h)^{2A+2}\bigg)\bigg(\frac{Y}{HW}\sum_{\substack{h\leq H}}\min\bigg\{\frac{Y}{HW},\frac{1}{\|\alpha
hx\|}\bigg\}\bigg)\\
\ll&_A(\log D)^{2^{2A+2}}Y\sum_{\substack{1\leq x\leq
X'H}}\tau(x)\min\bigg\{\frac{Y}{HW},\frac{1}{\|\alpha x\|}\bigg\}.
\end{align*}
Finally,
\begin{align*}
&\sum_{\substack{1\leq x\leq X'H}}\tau(x)\min\bigg\{\frac{Y}{HW},\frac{1}{\|\alpha
x\|}\bigg\}\\
\ll&\bigg(\sum_{\substack{1\leq x\leq X'H}}\tau(x)^2\bigg)^{1/2}\bigg(\frac{Y}{HW}\sum_{\substack{1\leq x\leq X'H}}\min\bigg\{\frac{X'Y}{xW},\frac{1}{\|\alpha
x\|}\bigg\}\bigg)^{1/2}\\
\ll&\bigg(\frac{X^2Y^2(\log(DXq))^{4}}{W^2}\bigg(\frac{1}{q}+\frac{DW}{Y}+\frac{qW}{XY}\bigg)\bigg)^{1/2}.
\end{align*}
\end{proof}

Define
$$
\tau_k(x)=\{(d_1,d_2,\ldots,d_k):\, d_1d_2\cdots d_k\mid x\}.
$$
Let $G(x)$ be an arbitrary complex function over $\N$. Consider

\medskip\noindent {\bf Type I sums}{\it
$$
\sum_{\substack{M<m\leq M_1\\ L<l\leq L_1\\ P\leq ml\leq
P_1}}a_mG(ml)\qquad\text{and}\qquad\sum_{\substack{M<m\leq M_1\\
L<l\leq L_1\\ P\leq ml\leq P_1}}a_m(\log l)G(ml)
$$
where $M_1\leq 2M$, $L_1\leq 2L$, $|a_m|\leq\tau_5(m)\log P$,}

\medskip\noindent
and

\medskip\noindent
{\bf Type II sums}{\it
$$
\sum_{\substack{M<m\leq M_1\\
L<l\leq L_1\\ P\leq ml\leq P_1}}a_mb_lG(ml)
$$
where $M_1\leq 2M$, $L_1\leq 2L$, $|a_m|\leq\tau_5(m)\log P$,
$|b_l|\leq\tau_5(l)\log P$.}

\medskip
The following Lemma is due to Heath-Brown \cite{HeathBrown82}:
\begin{Lem}
\label{heathbrown} Let $P, P_1, u, v, z$ be positive integers
satisfying $2<P<P_1\leq 2P$, $2\leq u<v\leq z\leq P$, $u^2\leq z$,
$128uz^2\leq P_1$, $2^{18}P_1\leq v^3$. Then we may decompose the
sum
$$
\sum_{P<n\leq P_1}\Lambda(n)G(n)
$$
into $O((\log P)^6)$ sums, each of which is either of type I with $L\geq z$ or of type II with $u\leq L\leq v$.
\end{Lem}

\begin{Lem}
\label{spmminor}
For any $\alpha\in\m$,
$$
S_{n,z_0}^{\pm}(\alpha)\ll n(\log n)^{6^8-B/4}.
$$
\end{Lem}
\begin{proof}
Clearly
$$
S_{n,z_0}^{\pm}(\alpha)=\sum_{\substack{n^{0.99}\leq p\leq n\\
p\equiv b\pmod{W}}}e(\alpha(p-b)/W)\log p\sum_{d\mid
(p+2,P(z_0))}\lambda_D^\pm(d)+O(n^{0.995}).
$$
And notice that for any $x\leq n$,
$$
\bigg|\sum_{d\mid(x+2,P(z_0))}\lambda_D^{\pm}(d)\bigg|\leq\tau(x+2)\ll_\epsilon n^\epsilon.
$$
So it suffices to estimate the sum
\begin{equation}
\label{minorm} \sum_{\substack{n'\leq x\leq n\\ x\equiv
b\pmod{W}}}\Lambda(x)e(\alpha(x-b)/W)\sum_{d\mid(x+2,P(z_0))}\lambda_D^{\pm}(d),
\end{equation}
where $n'\geq n/2$. Since $\alpha\in\m$, there exist $1\leq a\leq q$
with $(a,q)=1$ and $(\log n)^B\leq q\leq n(\log n)^{-B}$ such that
$|\alpha q-a|\leq (\log n)^B/n$. Applying Lemma \ref{heathbrown}
with $u=n^{0.17}$, $v=n^{0.34}$ and $z=n^{0.35}$, the sum
(\ref{minorm}) can be decomposed into $O((\log n)^6)$ type I sums
$$
\sum_{\substack{M<m\leq M_1\\ L<l\leq L_1\\ ml\equiv b\pmod{W}\\ n'\leq ml\leq
n}}a_me(\alpha(ml-b)/W)\sum_{d\mid(ml+2,P(z_0))}\lambda_D^{\pm}(d)
$$
and
$$
\qquad\sum_{\substack{M<m\leq M_1\\
L<l\leq L_1\\ ml\equiv b\pmod{W}\\ n'\leq ml\leq n}}a_m(\log
l)e(\alpha(ml-b)/W)\sum_{d\mid(ml+2,P(z_0))}\lambda_D^{\pm}(d)
$$
with $L\geq n^{0.35}$, and type II sums
$$
\sum_{\substack{M<m\leq M_1\\
L<l\leq L_1\\ ml\equiv b\pmod{W}\\ n'\leq ml\leq
n}}a_mb_le(\alpha(ml-b)/W)\sum_{d\mid(ml+2,P(z_0))}\lambda_D^{\pm}(d)
$$
with $n^{0.17}\leq L\leq n^{0.34}$. Noting that
$\lambda_D^{\pm}(d)=0$ whenever $d\geq D$, in view of
(\ref{expsum1}) with $A=5$, these type II sums are all $\ll n(\log
n)^{2^{13}-B/4}$. And by (\ref{expsum2}), all type I sums are $\ll
n(\log n)^{6^7-B/4}$.

\end{proof}

\section{The Major Arcs}
\setcounter{equation}{0} \setcounter{Thm}{0} \setcounter{Lem}{0}
\setcounter{Cor}{0}

Define
$$
\Delta(x;q):=\max_{\substack{1\leq r\leq q\\
(r,q)=1}}\bigg|\sum_{\substack{p\leq x\\ p\equiv r\pmod{q}}}\log
p-\frac{x}{\phi(q)}\bigg|.
$$
The well-known Bombieri-Vinogradov theorem asserts that for any $A>0$
\begin{equation}
\label{bv} \sum_{q\leq n^{1/2-\epsilon}}\max_{x\leq
n}\Delta(x;q)\ll_{A,\epsilon}\frac{n}{(\log n)^A}.
\end{equation}

Define
$$
\phi_2(q):=q\prod_{\substack{2<p\mid q\\
}}\bigg(1-\frac{2}{p}\bigg).
$$
\begin{Lem}
\label{spmFf}
$$
S_{n,z_0}^+(0)\leq\frac{4e^{-\gamma}k_0\fS_1 n}{\phi_2(W)\log
n}(F(s)+O(e^{-s}(\log n)^{-1/3}))
$$
and
$$
S_{n,z_0}^-(0)\geq\frac{4e^{-\gamma}k_0\fS_1 n}{\phi_2(W)\log
n}(f(s)+O(e^{-s}(\log n)^{-1/3})),
$$
where $\gamma$ is Euler's constant,
$$
\fS_1=\prod_{p>2}\bigg(1-\frac{1}{(p-1)^2}\bigg)=0.6601\ldots
$$
and $s=\log D/\log z_0$.
\end{Lem}
\begin{proof}
\begin{align*}
S_{n,z_0}^+(0)
=&\sum_{\substack{d\mid P(z_0)\\ b\equiv -2\pmod{(d,W)}}}\lambda_D^+(d)\sum_{\substack{p\leq n\\ p\equiv b\pmod{W}\\
p\equiv -2\pmod{d}}}\log p\\
=&\sum_{\substack{d\mid P(z_0)\\
(d,W)=1}}\lambda_D^+(d)\bigg(\frac{n}{\phi(Wd)}+O(\Delta(n;Wd))\bigg)
\end{align*}
since $(W,b+2)=1$. Since $\lambda_D^+(d)$ vanishes for $d\geq D$, by
(\ref{bv}) we have,
\begin{align*}
S_{n,z_0}^+(0) =\frac{n}{\phi(W)}\sum_{\substack{d\mid P(z_0)\\
(d,W)=1}}\frac{\lambda_D^+(d)}{\phi(d)}+O\bigg(\frac{n}{(\log
n)^{5B}}\bigg).
\end{align*}
Applying (\ref{rosserF}),
\begin{align*}
\sum_{\substack{d\mid P(z_0)\\
(d,W)=1}}\frac{\lambda_D^+(d)}{\phi(d)} \leq\prod_{p\leq z_0,\
p\nmid W}\bigg(1-\frac{1}{p-1}\bigg)(F(s)+O(e^{-s}(\log D)^{-1/3})).
\end{align*}
Similarly,
\begin{align*}
S_{n,z_0}^-(0) \geq\frac{n}{\phi(W)}\prod_{p\leq z_0,\ p\nmid
W}\bigg(1-\frac{1}{p-1}\bigg)(f(s)+O(e^{-s}(\log n)^{-1/3})).
\end{align*}
Finally, by the Mertens theorem,
\begin{align*}
\prod_{p\leq z_0,\ p\nmid W}\bigg(1-\frac{1}{p-1}\bigg)^{-1}=
\frac{\phi_2(W)}{2\phi(W)}\prod_{2<p\leq
z_0}\bigg(1-\frac{1}{p-1}\bigg)^{-1}=&
\frac{\phi_2(W)}{4\phi(W)}(\fS_1^{-1}e^{\gamma}\log{z_0}+O(1)).
\end{align*}
\end{proof}

Let $m=(n-b)/W$. Define $\Lambda_*(x)=\log x$ or $0$ according to whether $x$ is prime.
\begin{Lem}
\label{saq} Suppose that $1\leq a\leq q\leq (\log n)^B$ and
$(a,q)=1$. Then we have
\begin{align}
\bigg|S_{n,z_0}(a/q)-\frac{\1_{(W,q)=1}\mu(q)\tau^*(a,q)4e^{-\gamma}k_0\fS_1Wm}{\phi_2(Wq)\log(Wm+b)}\bigg|
\leq\frac{5e^{-\gamma}k_0\fS_1(F(s)-f(s))Wm}{\phi_2(W)\log(Wm+b)},
\end{align}
where $\1_{(W,q)=1}=1$ or $0$ according to whether $(W,q)=1$,
$$
\tau^*(a,q)=\sum_{\substack{d\mid q\\ (d,q/d)=1}}e(ar_d/q),
$$
and $1\leq r_d\leq q$ is the unique integer $r$ such that $Wr\equiv
-b\pmod{d}$ and $Wr\equiv -b-2\pmod{q/d}$.
\end{Lem}
\begin{proof} Clearly
\begin{align*}
&\sum_{\substack{1\leq x\leq m\\ (Wx+b+2,P(z_0))=1}}\Lambda_*(Wx+b)e(ax/q)
=&\sum_{\substack{1\leq r\leq q\\ (Wr+b,q)=1\\ (Wr+b+2,q)=1}}e(ar/q)\sum_{\substack{1\leq p\leq Wm+b\\
p\equiv Wr+b\pmod{Wq}\\ (p+2,P(z_0))=1}}\log p+O(\log^{B+1}n).
\end{align*}
Notice that
\begin{align*}
&\sum_{\substack{1\leq p\leq Wm+b\\
p\equiv Wr+b\pmod{Wq}}}\log p\sum_{d\mid(p+2,P(z_0))}\lambda_D^-(d)\\ \leq&\sum_{\substack{1\leq p\leq Wm+b\\
p\equiv Wr+b\pmod{Wq}}}\log p\sum_{d\mid (p+2,P(z_0))}\mu(d)\\\leq&\sum_{\substack{1\leq p\leq Wm+b\\
p\equiv Wr+b\pmod{Wq}}}\log p\sum_{d\mid(p+2,P(z_0))}\lambda_D^+(d).
\end{align*}
From the proof of Lemma \ref{spmFf}, we know that
\begin{align*}
&\bigg|\sum_{\substack{1\leq p\leq Wm+b\\
p\equiv Wr+b\pmod{Wq}\\ (p+2,P(z_0))=1}}\log p-
\frac{e^{-\gamma}k_0\fS_1Wm}{\phi_2(Wq)\log(Wm+b)}\bigg|\\
\leq&\frac{1.1e^{-\gamma}k_0\fS_1(F(s)-f(s))Wm}{\phi_2(Wq)\log(Wm+b)}.
\end{align*}
By noting that $W$ is even and $(W,b(b+2))=1$,
\begin{align*}
\sum_{\substack{1\leq r\leq q\\ (Wr+b,q)=1\\ (Wr+b+2,q)=1}}e(ar/q)
=&\sum_{d_1,d_2\mid q}\mu(d_1)\mu(d_2)\sum_{\substack{1\leq r\leq q\\ d_1\mid Wr+b\\ d_2\mid Wr+b+2}}e(ar/q)\\
=&\sum_{\substack{d_1d_2=q\\ (d_1,d_2)=1\\ (d_1,W)=(d_2,W)=1}}\mu(d_1)\mu(d_2)e(ar_{d_1}/q)\\
=&\begin{cases}
\mu(q)\tau^*(a,q)&\text{if }(W,q)=1,\\
0&\text{otherwise}.
\end{cases}
\end{align*}
Furthermore, we have
$$
|\{1\leq r\leq q:\, ((Wr+b)(Wr+b+2),q)=1\}|=q\prod_{\substack{p\mid q,\ p\nmid W}}\bigg(1-\frac{2}{p}\bigg)=\frac{\phi_2(Wq)}{\phi_2(W)}.
$$
All are done.
\end{proof}

\begin{Lem}
\label{spmalpha}
Suppose that $1\leq a\leq q\leq (\log n)^B$ and $(a,q)=1$. Then for any
$\alpha\in\M_{a,q}$,
\begin{align}
S_{n,z_0}^\pm(\alpha)=\frac{S_{n,z_0}^\pm(a/q)}{m}\sum_{1\leq y\leq
m}e(\theta y)+ O\bigg(\frac{m}{(\log m)^{3B}}\bigg),
\end{align}
where $\theta=\alpha-a/q$.
\end{Lem}
\begin{proof}
By the partial summation,
\begin{align*}
&S_{n,z_0}^\pm(\alpha)\\
=&e(\theta m)\sum_{\substack{1\leq x\leq m}}\Lambda_*(Wx+b)e(ax/q)\sum_{d\mid (Wx+b+2,P(z_0))}\lambda_D^{\pm}(d)\\
&-\sum_{y\leq m-1}(e(\theta(y+1))-e(\theta y))\sum_{\substack{1\leq
x\leq y}}\Lambda_*(Wx+b)e(ax/q)\sum_{d\mid
(Wx+b+2,P(z_0))}\lambda_D^{\pm}(d).
\end{align*}
Recalling that $(W,b+2)=1$, write
\begin{align*}
&\sum_{\substack{1\leq x\leq y}}\Lambda_*(Wx+b)e(ax/q)\sum_{d\mid (Wx+b+2,P(z_0))}\lambda_D^{\pm}(d)\\
=&\sum_{\substack{d\mid P(z_0)\\ (d,W)=1\\ d\leq D}}\lambda_D^\pm(d)\sum_{\substack{1\leq x\leq y\\
Wx\equiv-b-2\pmod{d}}}\Lambda_*(Wx+b)e(ax/q)\\
=&\sum_{\substack{d\mid P(z_0)\\ (d,W)=1\\ d\leq D}}\lambda_D^\pm(d)\bigg(\sum_{\substack{1\leq r\leq q\\ (Wr+b,q)=1\\ Wr\equiv-b-2\pmod{(d,q)}}}e(ar/q)\sum_{\substack{1\leq x\leq y\\
x\equiv r\pmod{q}\\
Wx\equiv-b-2\pmod{d}}}\Lambda_*(Wx+b)+O(\log^{B+1}n)\bigg).
\end{align*}
Now
\begin{align*}
&\sum_{\substack{1\leq x\leq y\\
x\equiv r\pmod{q}\\
Wx\equiv-b-2\pmod{d}}}\Lambda_*(Wx+b)=\frac{Wy+b}{\phi(W[d,q])}+O(\Delta(Wy+b;W[d,q])).
\end{align*}
Notice that for any $d'$ with $q\mid d'$,
$|\{d:\,[d,q]=d'\}|\leq\tau(q)$. Hence
\begin{align*}
&\sum_{\substack{1\leq x\leq y}}\Lambda_*(Wx+b)e(ax/q)\sum_{d\mid (Wx+b+2,P(z_0))}\lambda_D^{\pm}(d)\\
=&\sum_{\substack{d\mid P(z_0)\\ (d,W)=1\\ d\leq
D}}\lambda_D^\pm(d)\sum_{\substack{1\leq r\leq q\\ (Wr+b,q)=1\\
Wr\equiv-b-2\pmod{(d,q)}}}e(ar/q)\frac{Wy+b}{\phi(W[d,q])}
+O\bigg(\frac{Wy+b}{(\log(Wy+b))^{5B}}\bigg).
\end{align*}
So
\begin{align*}
S_{n,z_0}^\pm(\alpha) =W\sum_{1\leq y\leq m}e(\theta
y)\sum_{\substack{d\mid P(z_0)\\ (d,W)=1\\ d\leq
D}}\lambda_D^\pm(d)\sum_{\substack{1\leq r\leq q\\ (Wr+b,q)=1\\
Wr\equiv-b-2\pmod{(d,q)}}}\frac{e(ar/q)}{\phi(W[d,q])}
+O\bigg(\frac{m}{(\log m)^{3B}}\bigg).
\end{align*}
Setting $\theta\rightarrow0$ in the above equation, we obtain that
\begin{align*}
W\sum_{\substack{d\mid P(z_0)\\ (d,W)=1\\ d\leq
D}}\lambda_D^\pm(d)\sum_{\substack{1\leq r\leq q\\ (Wr+b,q)=1\\
Wr\equiv-b-2\pmod{(d,q)}}}\frac{e(ar/q)}{\phi(W[d,q])}
=\frac{S_{n,z_0}^\pm(a/q)}{m} +O\bigg(\frac{1}{(\log m)^{3B}}\bigg).
\end{align*}
\end{proof}
Combining Lemmas \ref{spm}, \ref{spmminor}, \ref{spmFf}, \ref{saq}
and \ref{spmalpha}, we get

\begin{Lem}
\label{sumexp}
Suppose that $1\leq a\leq q\leq (\log n)^B$ and $(a,q)=1$. Then for any
$\alpha\in\M_{a,q}$,
\begin{align}
\bigg|S_{n,z_0}(\alpha)-\frac{\1_{(W,q)=1}\mu(q)\tau^*(a,q)4e^{-\gamma}k_0\fS_1W}{\phi_2(Wq)\log
n}\sum_{1\leq y\leq m}e(\theta y)\bigg|\leq
\frac{15e^{-\gamma}k_0\fS_1(F(s)-f(s))n}{\phi_2(W)\log n},
\end{align}
where $\theta=\alpha-a/q$. Furthermore, for any $\alpha\in\m$,
\begin{align}
|S_{n,z_0}(\alpha)|\leq
\frac{5e^{-\gamma}k_0\fS_1(F(s)-f(s))n}{\phi_2(W)\log n}.
\end{align}
\end{Lem}

\begin{Lem}
\label{primediff}
\begin{align}
\sum_{\substack{p_1,p_2\leq n\\
(p_i+2,P(z_0))=1\\ p_i\equiv b\pmod{W}\\ p_2-p_1=WM}}1
\ll\frac{k_0^2nW}{\phi_2(W)^2\log^4 n}\prod_{\substack{p\mid M\\
p\nmid W}}\bigg(1+\frac{2}{p}\bigg)\prod_{\substack{p\mid
(WM+2)(WM-2)\\ p\nmid W}}\bigg(1+\frac{1}{p}\bigg).
\end{align}
\end{Lem}
\begin{proof}
Let $z_1=n^{1/10}$. Let $\omega_1$ and $\omega_2$ be two
multiplicative functions satisfying that
$$
\omega_1(p)=\begin{cases} 4&\quad\text{ if
}p<z_0\text{ and }p\nmid{WM(WM-2)(WM+2)},\\
3&\quad\text{ if
}p<z_0\text{ and }p\mid{(WM-2)(WM+2)},\ p\nmid W,\\
2&\quad\text{ if
}p<z_0,\ p\mid{M}\text{ and }p\nmid W,\\
0&\quad\text{otherwise},\\
\end{cases}
$$
and
$$
\omega_2(p)=\begin{cases} 2&\quad\text{ if
}z_0\leq p<z_1\text{ and }p\nmid{WM},\\
1&\quad\text{ if
}z_0\leq p<z_1,\ p\mid{M}\text{ and }p\nmid W,\\
0&\quad\text{otherwise},
\end{cases}
$$
for prime $p$. And for $1\leq i\leq 2$, let $g_i$  be the
multiplicative functions with
$$
g_i(p)=\frac{\omega_i(p)}{p}\bigg(1-\frac{\omega_i(p)}{p}\bigg)^{-1}
$$
for prime $p$, and let
$$
G_1^{(i)}(z)=\sum_{\substack{l\mid P(z)\\ l<z}}g_i(l).
$$
Define
$$
\lambda_1(d)=\frac{d}{\omega_1(d)}\sum_{\substack{l\mid P(z_0)\\
d\mid l<z_0}}\frac{\mu(l/d)\mu(l)g_1(l)}{G_1^{(1)}(z_0)}
$$
for $d\mid P(z_0)$ and
$$
\lambda_2(d)=\frac{d}{\omega_2(d)}\sum_{\substack{l\mid P(z_0,z_1)\\
d\mid l<z_1}}\frac{\mu(l/d)\mu(l)g_2(l)}{G_1^{(2)}(z_1)}
$$
for $d\mid P(z_0,z_1)=\prod_{z_0\leq p<z_1}p$. Then
$\lambda_1(1)=\lambda_2(1)=1$. Therefore
\begin{align*}
&\sum_{\substack{x_1,x_2\leq n/W\\
(Wx_i+b+2,P(z_0))=1\\ (Wx_i+b,P(z_1))=1\\ x_2-x_1=M}}1\\
\leq&\sum_{\substack{x\leq n/W}}\bigg(\sum_{\substack{d\mid P(z_0)\\
d\mid  (Wx+b)(Wx+WM+b)\\
d\mid (Wx+b+2)(Wx+WM+b+2)}}\lambda_1(d)\bigg)^2
\bigg(\sum_{\substack{d\mid P(z_0,z_1)\\
d\mid (Wx+b)(Wx+WM+b)}}\lambda_2(d)\bigg)^2\\
=&\sum_{\substack{d_1,d_2\mid P(z_0)\\ d_3,d_4\mid
P(z_0,z_1)}}\lambda_1(d_1)\lambda_1(d_2)\lambda_2(d_3)\lambda_2(d_4)\sum_{\substack{x\leq n/W\\
[d_3,d_4]\mid (Wx+b)(Wx+WM+b)\\
[d_1,d_2],[d_3,d_4]\mid (Wx+b+2)(Wx+WM+b+2)}}1\\
=&\sum_{\substack{d_1,d_2\mid P(z_0)\\ d_3,d_4\mid
P(z_0,z_1)}}\lambda_1(d_1)\lambda_1(d_2)\lambda_2(d_3)\lambda_2(d_4)\omega_1([d_1,d_2])\omega_2([d_3,d_4])\bigg(\frac{n/W}{[d_1,d_2][d_3,d_4]}+O(1)\bigg).
\end{align*}
By Selberg's sieve method, we know that
$|\lambda_1(d)|,|\lambda_2(d)|\leq 1$ and
\begin{align*}
&\sum_{\substack{d_1,d_2\mid P(z_0)\\ d_3,d_4\mid
P(z_0,z_1)}}\lambda_1(d_1)\lambda_1(d_2)\lambda_2(d_3)\lambda_2(d_4)\frac{\omega_1([d_1,d_2])\omega_2([d_3,d_4])}{[d_1,d_2][d_3,d_4]}\\
=&\bigg(\sum_{d_1,d_2\mid
P(z_0)}\lambda_1(d_1)\lambda_1(d_2)\frac{\omega_1([d_1,d_2])}{[d_1,d_2]}\bigg)\bigg(\sum_{d_3,d_4\mid
P(z_0,z_1)}\lambda_2(d_3)\lambda_2(d_4)\frac{\omega_2([d_3,d_4])}{[d_3,d_4]}\bigg)\\
=&\frac{1}{G_1^{(1)}(z_0)}\frac{1}{G_1^{(2)}(z_1)}\ll\prod_{p
<z_0}\bigg(1-\frac{\omega_1(p)}{p}\bigg)\prod_{z_0\leq
p<z_1}\bigg(1-\frac{\omega_2(p)}{p}\bigg).
\end{align*}
Thus
\begin{align*}
&\sum_{\substack{x_1,x_2\leq n/W\\
(Wx_i+b+2,P(z_0))=1\\ (Wx_i+b,P(z_1))=1\\
x_2-x_1=M}}1\\
\ll&\frac{n}{W}\prod_{p\mid
P(z_0)}\bigg(1-\frac{\omega_1(p)}{p}\bigg)\prod_{p\mid
P(z_0,z_1)}\bigg(1-\frac{\omega_2(p)}{p}\bigg)\\
\ll&\frac{n}{W(\log z_0)^2(\log z_1)^2}\prod_{\substack{
p\mid W}}\bigg(1+\frac{4}{p}\bigg)\prod_{\substack{p\mid M\\
p\nmid W}}\bigg(1+\frac{2}{p}\bigg)\prod_{\substack{p\mid
(WM+2)(WM-2)\\ p\nmid W}}\bigg(1+\frac{1}{p}\bigg)\\
\ll&\frac{k_0^2nW}{\phi_2(W)^2\log^4 n}\prod_{\substack{p\mid M\\
p\nmid W}}\bigg(1+\frac{2}{p}\bigg)\prod_{\substack{p\mid
(WM+2)(WM-2)\\ p\nmid W}}\bigg(1+\frac{1}{p}\bigg).
\end{align*}
\end{proof}

\begin{Lem}
Then
\begin{align}
\label{equalsum} \sum_{\substack{p_1,p_2,p_3,p_4\leq n\\
(p_i+2,P(z_0))=1\\ p_i\equiv b\pmod{W}\\ p_1+p_4=p_2+p_3}}1
\ll\frac{k_0^4Wn^3}{\phi_2(W)^4\log^8 n}.
\end{align}
\end{Lem}
\begin{proof}
Applying Lemma \ref{primediff},
\begin{align*}
\sum_{\substack{p_1,p_2,p_3,p_4\leq n\\
(p_i+2,P(z_0))=1\\ p_i\equiv b\pmod{W}\\
p_1+p_4=p_2+p_3}}1\leq&\sum_{\substack{2<M\leq n/W}}\bigg(\sum_{\substack{p_1,p_2\leq n\\
(p_i+2,P(z_0))=1\\ p_i\equiv b\pmod{W}\\
p_1+WM=p_2}}1\bigg)^2+O(n^2)\\
\ll&\frac{k_0^4n^2W^2}{\phi_2(W)^4\log^8 n}\sum_{\substack{2<M\leq
n/W}}\prod_{\substack{p\mid M\\ p\nmid
W}}\bigg(1+\frac{2}{p}\bigg)^2\prod_{\substack{p\mid (WM-2)(WM+2)\\
p\nmid W}}\bigg(1+\frac{1}{p}\bigg)^2.
\end{align*}
By the H\"older inequality,
\begin{align*}
&\sum_{\substack{2<M\leq n/W}}\prod_{\substack{p\mid M\\ p\nmid
W}}\bigg(1+\frac{2}{p}\bigg)^2
\prod_{\substack{p\mid (WM-2)(WM+2)\\ p\nmid W}}\bigg(1+\frac{1}{p}\bigg)^2\\
\leq&\bigg(\sum_{\substack{2<M\leq n/W}}\prod_{\substack{p\mid
M}}\bigg(1+\frac{2}{p}\bigg)^6\bigg)^{1/3}
\prod_{j=1}^2\bigg(\sum_{\substack{2<M\leq
n/W}}\prod_{\substack{p\mid WM+2(-1)^j\\
p\not=2}}\bigg(1+\frac{1}{p}\bigg)^6\bigg)^{1/3}.
\end{align*}
Since $(1+p^{-1})^6\leq 1+24p^{-1}$,
\begin{align*}
\sum_{\substack{2<M\leq n/W}}\prod_{\substack{p\mid WM\pm 2\\
p\not=2}}\bigg(1+\frac{1}{p}\bigg)^6 \leq&\sum_{\substack{2<M\leq
n/W}}
\sum_{\substack{d\mid WM\pm 2\\ 2\nmid d}}\frac{\tau(d)^{12}}{d}\\
=&\sum_{\substack{d\mid n\pm 2\\
(d,W)=1}}\frac{\tau(d)^{12}}{d}\sum_{\substack{2<M\leq n/W\\ d\mid
WM\pm 2}}1 \ll\frac{n}{W}.
\end{align*}
And by \cite[Lemma 14]{Green02}, we have
$$
\sum_{\substack{2<M\leq n/W}}\prod_{\substack{p\mid
M}}\bigg(1+\frac{2}{p}\bigg)^6\leq \sum_{\substack{2<M\leq
n/W}}\prod_{\substack{p\mid M}}\bigg(1+\frac{1}{p}\bigg)^{12}\ll
\frac{n}{W}.
$$
\end{proof}

\section{Proof of Theorem \ref{chengoldbach}}
\setcounter{equation}{0} \setcounter{Thm}{0} \setcounter{Lem}{0}
\setcounter{Cor}{0}

First, let us introduce Green and Tao's enveloping sieve. Let $N$ be
a large integer. Suppose that $a_1,\ldots,a_k,b_1,\ldots,b_k$ be
integers with $|a_i|,|b_i|\leq N$. We say
$$
\F(x):=\prod_{i=1}^k(a_ix+b_i).
$$
is a $k$-linear form.

For every integer $q\geq 1$, define
$$
\gamma_\F(p):=q^{-1}|\{1\leq x\leq q:\, (\F(x),q)=1\}|>0.
$$
Let
$$
X_{R!}(x)=\{x\in\Z:\, (\F(x),R!)=1\},
$$
where $1\leq R\leq N$.
\begin{Lem}[{\cite[Proposition 3.1]{GreenTao06}}]
\label{envsieve} There exists a non-negative function
$\beta_R:\Z\to\R$ satisfying the following properties:

\medskip\noindent(i)
\begin{equation}
\label{betalower} \beta_R(x)\gg_k\mathfrak{S}_\F^{-1}\log^k R
\,\1_{X_{R!}}(x)
\end{equation}
for all integers $n$, where
$$
\fS_\F:=\prod_{p}\frac{\gamma_\F(p)}{(1-1/p)^k}.
$$

\medskip\noindent(ii)
\begin{equation}
\label{beteupper}
\beta_R(x)\ll_{k,\epsilon}N^\epsilon
\end{equation}
for all $1\leq x\leq N$ and $\epsilon>0$.

\medskip\noindent(iii)
\begin{equation}
\beta_R(x)=\sum_{q\leq R^2}\sum_{\substack{1\leq a\leq q\\ (a,q)=1}}w(a/q)e(-ax/q),
\end{equation}
where $w(a/q)=w_R(a/q)$ satisfies $w(1)=1$ and
$$
|w(a/q)|\ll_{k,\epsilon} q^{\epsilon-1}
$$
for all $1\leq a\leq q\leq R^2$ with $(a,q)=1$.

\medskip\noindent(iv)
For $1\leq a\leq q\leq R^2$ with $(a,q)=1$, if $q$ is not square-free, or $\gamma(q)=1$ and $q>1$, then
$w(a/q)=0$.
\end{Lem}
Green and Tao also established a restriction theorem for $\beta_R$:
\begin{Lem}[{\cite[Proposition 4.2]{GreenTao06}}]
Let $R$, $N$ be large numbers such that $1\leq R\leq N^{1/10}$. Let
$k$, $\F$, $\beta_R$ as defined in Lemma \ref{envsieve}. Suppose
that $\{u_i\}_{i=1}^N$ is an arbitrary sequence of complex numbers.
Then for any $\rho>2$,
\begin{equation}\label{restriction}\bigg(\sum_{r\in\Z_N}\bigg|\frac{1}{N}\sum_{1\leq x\leq
N}u_x\beta_R(x)e(-xr/N)\bigg|^\rho\bigg)^{1/\rho}
\ll_{\rho,k}\bigg(\frac{1}{N}\sum_{1\leq x\leq
N}|u_x|^2\beta_R(x)\bigg)^{1/2}. \end{equation}
\end{Lem}

The following lemma can be derived by a trivial modification of
Chen's original proof in \cite{Chen73}:
\begin{Lem}
\label{chen}
Suppose that $W$ is a positive integer. Then there exists a function $n_0(W)$ such that for every $1\leq b\leq W$ with $(b(b+2),W)=1$ and $n\geq n_0(W)$,
$$
|\{p\leq n:\, p\equiv b\pmod{W},\ p+2\in\P_2\text{ and
}(p+2,P(n^{1/10}))=1\}|\geq \frac{C_1}{\phi_2(W)}\frac{n}{(\log
n)^2},
$$
where $C_1$ is an absolute constant.
\end{Lem}

\begin{Lem}
\label{nW}
Suppose that $n_0(W)$ is an increasing positive function for all positive integers $W$, and
$G(n)$ is an increasing positive function with $\lim_{n\to\infty}G(n)=\infty$. Then there exists an increasing
positive function $W_0(n)$ for sufficiently large $n$ such that $n\geq n_0(W(n))$, $W_0(n)\leq G(n)$ and $\lim_{n\to\infty}W_0(n)=\infty$.
\end{Lem}
\begin{proof}
Let
$$
W_0(n)=\max\{W\leq G(n):\, n_0(W)\leq n\}.
$$
Clearly $W_0(n)$ is well-defined for sufficiently large $n$. Assume on the contrary that
$\lim_{n\to\infty}W_0(n)<\infty$, i.e., there exists an integer $W'$ such that $W_0(n)<W'$ for all $n$.
Let $n'\geq n_0(W')$ be an integer such that $G(n')\geq W'$. Obviously such $n'$ exists. But
$W_0(n')\geq W'$ now. This leads to a contradiction.
\end{proof}

Let $C_2$ be the implied constant in (\ref{betalower}) with $k=2$.
Let
$$
\varpi=\frac{\min\{C_1C_2,1\}}{10000}.
$$

Let $C_3$ be the implied constant in (\ref{restriction}) with
$\rho=12/5$ and $k=2$. And let $C_4$ be the implied constant in
(\ref{equalsum}).
\begin{Lem}
\label{kappa} There exists $0<\kappa\leq\varpi$ such that we can
choose $0<\delta,\epsilon\leq 1$ with
$\epsilon^{6C_3^{12/5}\delta^{-12/5}+60C_4\delta^{-4}}\geq\kappa$
and
$$
3072\epsilon^2(C_3\delta^{-12/5}+5C_4\delta^{-4})+
72C_3^{24/13}C_4^{3/13}\delta^{1/13}\leq\varpi^6.
$$
\end{Lem}
\begin{proof}
Let
$$
\delta=\min\bigg\{\frac{\varpi^{78}}{144^{13}C_3^{24}C_4^{3}},1\bigg\}
$$
and
$$
\epsilon=\min\bigg\{\frac{\varpi^{3}\delta^{2}}{192(C_3\delta^{-12/5}+5C_4\delta^{-4})^{1/2}},1\bigg\}.
$$
Clearly $3072\epsilon^2(C_3\delta^{-12/5}+5C_4\delta^{-4})+
72C_3^{24/13}C_4^{3/13}\delta^{1/13}\leq\varpi^6$. So we may
arbitrarily choose a $\kappa$ satisfying
$$
\kappa\leq\min\{\epsilon^{6C_3^{12/5}\delta^{-12/5}+60C_4\delta^{-4}},\varpi\}.
$$
\end{proof}

Let $\kappa$ be a small constant satisfying the requirements of
Lemma \ref{kappa}. Notice that $f(s)$ is increasing, $F(s)$ is
decreasing and $F(s),f(s)=1+O(e^{-s})$. Choose a sufficiently large
$k_0$ satisfying that
$$
20(F(k_0/4)-f(k_0/4))\leq\kappa^2.
$$

Suppose that $n$ is a sufficiently large integer.
Let $w=w(n)$ be a positive function satisfying $P(w)\leq\log n$, $n\geq n_0(P(w))$ (where $n_0$ is defined in Lemma \ref{chen})
and $\lim_{n\to\infty}w(n)=\infty$. By Lemma \ref{nW}, such $w$ exists. Let $W=P(w)$. The following lemma can be easily verified:
\begin{Lem}
For any odd integer $n$ with $3\mid n$, there exist $1\leq
b_1,b_2,b_3\leq W$ with $(b_i(b_i+2),W)=1$ such that
$$
n\equiv b_1+b_2+b_3\pmod{W}.
$$
Furthermore, if $n\equiv 4\pmod{6}$, then there exist $1\leq
b_1,b_2\leq W$ with $(b_i(b_i+2),W)=1$ such that
$$
n\equiv b_1+b_2\pmod{W}.
$$
\end{Lem}

Now suppose that $n$ is odd. Suppose that $1\leq b_1,b_2,b_3\leq W$
are integers satisfying $(b_i(b_i+2),W)=1$ and $n\equiv
b_1+b_2+b_3\pmod{W}$. Let $n'=(n-b_1-b_2-b_3)/W$.

Let $N$ be a prime in the interval $[(1+\kappa^2/20)n/W,
(1+\kappa^2/10)n/W]$ and $R=N^{1/10}$. Thanks to the prime number
theorem, such a prime $N$ always exists whenever $n$ is sufficiently
large. Let $\F_i(x)=(Wx+b_i)(Wx+b_i+2)$ for $i=1,2$. Substituting
$N,\ R,\ \F_i$ to Lemma \ref{envsieve}, we get the desired functions
$\beta_i=\beta_{i,R}$ for $i=1,2$.

Let
$$
A_i= \{x\leq (n-b_i)/2W:\, Wx+b_i\text{ is Chen's prime and
}(Wx+b_i+2,P(n^{1/10}))=1\},
$$
for $i=1,2$, and define
$$
\a_i(x)=
\1_{A_i}(x)\frac{C_2\fS_1^{-1}}{1000}\frac{\phi_2(W)(\log(Wx+b_i))^2}{n}.
$$
By Lemma \ref{chen}, clearly we have
\begin{equation}
\label{a12sum}
\sum_{x}\a_i(x)\geq(1-\kappa)\frac{C_2\fS_1^{-1}}{1000}\frac{\phi_2(W)(\log
n)^2|A_i|}{n}\geq 6\varpi,
\end{equation}
whenever $n$ is sufficiently large.

On the other hand, it is easy to see that
$$
\fS_{\F_i}=\prod_{\substack{p\mid W\\
p>2}}\frac{1}{(1-1/p)^2}\prod_{p\nmid
W}\frac{1-2/p}{(1-1/p)^2}=\prod_{\substack{p\mid W\\
p>2}}\frac{p}{p-2}
\prod_{\substack{p>2}}\bigg(1-\frac{1}{(p-1)^2}\bigg)=\frac{W\fS_1}{\phi_2(W)}.
$$
Hence by (\ref{betalower}), for any $x$
\begin{equation}
\label{abeta} \beta_i(x)\geq \frac{C_2\phi_2(W)\fS_1^{-1}}{W}(\log
R)^2\1_{A_i}(x)\geq N\a_i(x).
\end{equation}

Let
$$
A_3=\{x\leq (n-b_3)/W:\, Wx+b_3\text{ is prime and
}(Wx+b_3+2,P(n^{1/k_0}))=1\}.
$$
and define
$$
\a_3(x)=\1_{A_3}(x)\frac{e^{\gamma}\phi_2(W)\log(Wx+b_3)\log
n}{4k_0\fS_1n}.
$$
In view of Lemmas \ref{spm} and \ref{spmFf}, we also have
\begin{equation}
\label{a3sum} \sum_{x}\a_3(x)=\frac{e^{\gamma}\phi_2(W)\log
n}{4k_0\fS_1n}S_{n,n^{1/k_0}}(0)\in[1-\kappa^2,1+\kappa^2].
\end{equation}

Below we identify the set $\{1,2,\ldots,N\}$ with the group
$\Z_N=\Z/N\Z$. If there exist $x_1\in A_1$, $x_2\in A_2$ and $x_3\in
A_3$ satisfying $x_1+x_2+x_3=n'$ in $\Z_N$, then the equality also
holds in $\Z$. In fact, since $x_1+x_2\leq n/W$ and $x_3<n/W$ in
$\Z$, we must have $x_1+x_2+x_3<n'+N$ in $\Z$.

For any function $f:\,\Z_N\to\C$, define
$$
\tilde{f}(r)=\sum_{x\in\Z_N}f(x)e(-xr/N).
$$
\begin{Lem}
\label{nuFour} Let $\nu_i=\beta_i/N$, then for any $r\in\Z_N$,
$$
|\tilde{\nu_i}(r)-\1_{r=0}|\leq C_5w^{-1/2},
$$
where $C_5$ is an absolute constant and $\1_{r=0}=1$ or $0$
according to whether $r=0$.
\end{Lem}
\begin{proof}
See \cite[Lemma 6.1]{GreenTao06}. (Notice that our definitions are a
little different from Green and Tao's.)
\end{proof}

\begin{Lem}
\label{aFour} For any $0\not=r\in\Z_N$,
$$
|\tilde{\a_3}(r)|\leq\frac{2}{w-2}+0.9\kappa^2.
$$
\end{Lem}
\begin{proof}
If $r/N\in\m$, then by Lemma \ref{sumexp}, we have
$$
|\tilde{\a_3}(r)|=\frac{e^\gamma\phi_2(W)\log
n}{k_0\fS_1n}|S_{n,n^{1/k_0}}(r/N)|\leq\frac{2}{5}\kappa^2.
$$
Suppose that there exist $1\leq a\leq q\leq (\log n)^B$ with
$(a,q)=1$ such that $r/N\in\M_{a,q}$. Then applying Lemma
\ref{sumexp},
$$
|\tilde{\a_3}(r)|\leq \frac{\phi_2(W)\log
n}{k_0\fS_1n}\bigg(\frac{\1_{(W,q)=1}|\tau^*(a,q)|k_0\fS_1W}{\phi_2(Wq)\log
n}\bigg|\sum_{1\leq y\leq m}e(\theta y)\bigg|+
\frac{7k_0\fS_1\kappa^2 n}{10\phi_2(W)\log n}\bigg),
$$
where $m=(n-b)/W$ and $\theta=r/N-a/q$. Recall that $\tau^*(a,q)=0$
whenever $(W,q)>1$. And if $a=q=1$, since $r\not=0$,
$$
\bigg|\sum_{1\leq y\leq m}e(yr/N)\bigg|\leq \bigg|\sum_{1\leq y\leq
N}e(yr/N)\bigg|+N-m\leq \frac{\kappa^2}{10} N.
$$
Suppose that $q>1$ and $(W,q)=1$. Then by noting $W=\prod_{p<w}p$,
$$
\frac{\phi_2(W)|\tau^*(a,q)|}{\phi_2(Wq)}\leq\frac{\phi_2(W)|\{d\mid
q:\, (d,q/d)=1\}|}{\phi_2(Wq)}\leq\frac{2}{w-2},
$$
since $q$ has at least one prime divisor not less than $w$. Finally,
we have $WN\leq1.1 n$.
\end{proof}

Suppose that $\delta,\epsilon>0$ are two small numbers to be chosen
later. For $1\leq i\leq 3$, let
$$
\RR_i=\{r\in\Z_N:\,|\tilde{a_i}(r)|>\delta\},
$$

$$
\B_i=\{x\in\Z_N:\,\|xr/N\|\leq \epsilon\},
$$
and define $\b_i=\1_{\B_i}/|B_i|$.
\begin{Lem}
\label{bohr}
$$
|\RR_3|\leq 60C_4\delta^{-4}.
$$
\end{Lem}
\begin{proof} By (\ref{equalsum}),
\begin{align*}
\sum_{r\in\Z_N}|\tilde{\a_3}(r)|^4=&N\sum_{\substack{x_1,x_2,x_3,x_4\in\Z_N\\ x_1+x_4=x_2+x_3}}\a_3(x_1)\a_3(x_2)\a_3(x_3)\a_3(x_4)\\
\leq&N\bigg(\frac{e^{\gamma}\phi_2(W)(\log n)^2}{4k_0\fS_1n}\bigg)^4\sum_{\substack{p_1,p_2,p_3,p_4\leq n\\
(p_i+2,P(n^{1/k_0}))=1\\
p_i\equiv b_3\pmod{W}\\ p_1+p_4=p_2+p_3}}1\\
\leq&C_4e^{4\gamma}(4\fS_1)^{-4}\frac{NW}{n}.
\end{align*}
Hence
$$
|\RR_3|\leq\sum_{r\in\Z_N}\delta^{-4}|\tilde{a_3}(r)|^4\leq
e^{4\gamma}(4\fS_1)^{-4}.
$$
\end{proof}
\begin{Lem}
\label{bohr}
$$
|\RR_1|,|\RR_2|\leq 6C_3^{12/5}\delta^{-12/5}
$$
provided that $w\geq C_5^{2}$.
\end{Lem}
\begin{proof} Suppose that $i\in\{1,2\}$.
Let $u_x=\a_i(x)/\nu_i(x)$ or $0$ according to whether
$\nu_i(x)\not=0$. By (\ref{abeta}), clearly $0\leq u_x\leq 1$. If
$w\geq C_5^2$, by Lemma 4.7 we have $\tilde{\nu_i}(0)\leq 2$.
Applying (\ref{restriction}),
\begin{align*}
\sum_{r\in\Z_N}\bigg|\sum_{1\leq x\leq
N}u_x\nu_i(x)e(-xr/N)\bigg|^{12/5}\leq& C_3^{12/5}\bigg(\sum_{1\leq
x\leq N}|u_x|^2\nu_i(x)\bigg)^{6/5}\\
\leq& C_3^{12/5}|\tilde{\nu_i}(0)|^{6/5}\\
\leq&6C_3^{12/5}.
\end{align*}
It follows that $|\RR_i|\leq 6C_3^{12/5}\delta^{-12/5}$.
\end{proof}

\begin{Lem}
\label{bohr}
$$
|\B_i|\geq\epsilon^{|\RR_i|}N.
$$
\end{Lem}
\begin{proof}
This is a simple application of the pigeonhole principle (cf.
\cite[Lemma 1.4]{Tao}).
\end{proof}
For two functions $f, g:\,\Z_N\to\C$. Define
$$
f*g(x)=\sum_{y\in\Z_N}f(y)g(x-y).
$$
It is easy to check that $\widetilde{(f*g)}=\tilde{f}\tilde{g}$. Let
$\a_i'=\a_i*\b_i*\b_i$ for $1\leq i\leq 3$.
\begin{Lem}
\label{a3upper} Suppose that $\epsilon^{|\RR_3|}\geq
(2/(w-2)+0.9\kappa^2)\kappa^{-1}$. Then for any $x\in\Z_N$,
$$
|\a_3'(x)|\leq\frac{1+2\kappa}{N}.
$$
\end{Lem}
\begin{proof}
\begin{align*}
|\a_3'(x)|=&|\a_3(x)*\b_3(x)*\b_3(x)|\\
\leq&\frac{1}{N}|\tilde{\a_3}(0)||\tilde{\b_3}(0)|^2+
\frac{1}{N}\sum_{r\not=0}|\tilde{\a_3}(r)\tilde{\b_3}(r)^2|\\
\leq&\frac{1}{N}|\tilde{\a_3}(0)||\tilde{\b_3}(0)|^2+
\frac{1}{N}\max_{r\not=0}|\tilde{\a_3}(r)|\sum_{r\in\Z_N}|\tilde{\b_3}(r)|^2\\
\leq&\frac{1+\kappa^2}{N}+
\frac{1}{|B_3|}\bigg(\frac{2}{w-2}+0.9\kappa^2\bigg),
\end{align*}
where we used Lemma \ref{aFour} in the last step. Thus the desired
result easily follows from Lemma \ref{bohr}.
\end{proof}
\begin{Lem}
\label{a12upper} Let $i\in\{1,2\}$. Suppose that
$\epsilon^{|\RR_i|}\geq C_5\kappa^{-1}w^{-1/2}$ and $C_5w^{-1/2}\leq
\kappa$. Then for any $x\in\Z_N$,
$$
|\a_i'(x)|\leq\frac{1+2\kappa}{N}.
$$
\end{Lem}
\begin{proof} Since $\nu_i(x)\geq\a_i(x)$, applying Lemma 4.7
\begin{align*}
|\a_i'(x)|\leq&|\nu_i(x)*\b_i(x)*\b_i(x)|\\
\leq&\frac{1}{N}|\tilde{\nu_i}(0)||\tilde{\b_i}(0)|^2+
\frac{1}{N}\max_{r\not=0}|\tilde{\nu_i}(r)|\sum_{r\in\Z_N}|\tilde{\b_i}(r)|^2\\
\leq&\frac{1+C_5w^{-1/2}}{N}+\frac{C_5w^{-1/2}}{|B_i|}.
\end{align*}
We are done.
\end{proof}
\begin{Lem}
\label{betaone} For $1\leq i\leq 3$,
$$
|1-\tilde{\b_i}(r)|\leq 16\epsilon^2
$$
for any $r\in\RR_i$.
\end{Lem}
\begin{proof}
See the proof of Lemma 6.7 of \cite{Green05}.
\end{proof}

\begin{Lem}
\label{threesum} Suppose that $w\geq C_5^2$. Then
\begin{align*}
&\bigg|\sum_{\substack{x_1,x_2,x_3\in\Z_N\\
x_1+x_2+x_3=n'}}\a_1'(x_1)\a_2'(x_2)\a_3'(x_3)-
\sum_{\substack{x_1,x_2,x_3\in\Z_N\\
x_1+x_2+x_3=n'}}\a_1(x_1)\a_2(x_2)\a_3(x_3)\bigg|\\
\leq&\frac{3072\epsilon^2(C_3^{12/5}\delta^{-12/5}+5C_4\delta^{-4})+
72C_3^{24/13}C_4^{3/13}\delta^{1/13}}{N}.
\end{align*}
\end{Lem}
\begin{proof}
Clearly
$$
\sum_{\substack{x_1,x_2,x_3\in\Z_N\\
x_1+x_2+x_3=n'}}\a_1(x_1)\a_2(x_2)\a_3(x_3)=\frac{1}{N}\sum_{r\in\Z_N}
\tilde{\a_1}(r)\tilde{\a_2}(r)\tilde{\a_3}(r)e(n'r/N).
$$
Hence
\begin{align*}
&\bigg|\sum_{\substack{x_1,x_2,x_3\in\Z_N\\
x_1+x_2+x_3=n'}}\a_1'(x_1)\a_2'(x_2)\a_3'(x_3)-
\sum_{\substack{x_1,x_2,x_3\in\Z_N\\
x_1+x_2+x_3=n'}}\a_1(x_1)\a_2(x_2)\a_3(x_3)\bigg|\\
\leq&\frac{1}{N}\sum_{r\in\Z_N}|\tilde{\a_1}(r)\tilde{\a_2}(r)\tilde{\a_3}(r)(1-\tilde{\b_1}(r)^2
\tilde{\b_2}(r)^2\tilde{\b_3}(r)^2)|.
\end{align*}

Since $w\geq C_5^2$, for $i=1,2$,
$$
|\tilde{a_i}(r)|\leq|\tilde{a_i}(0)|\leq|\tilde{\nu_i}(0)|\leq 2.
$$
We also have
$$
|\tilde{a_3}(r)|\leq|\tilde{a_3}(0)|\leq1+\kappa^2.
$$

If $r\in\RR_1\cap\RR_2\cap\RR_3$, then by Lemma \ref{betaone},
$$
|1-\tilde{\b_i}(r)^2|\leq|1-\tilde{\b_i}(r)|(1+|\tilde{\b_i}(r)|)\leq
|1-\tilde{\b_i}(r)|(1+|\tilde{\b_i}(0)|)\leq 32\epsilon^2.
$$
So
\begin{align*}
&|1-\tilde{\b_1}(r)^2\tilde{\b_2}(r)^2\tilde{\b_3}(r)^2|\\
\leq&|1-\tilde{\b_1}(r)^2|+ |\tilde{\b_1}(r)|^2|1-\tilde{\b_2}(r)^2|+
|\tilde{\b_1}(r)\tilde{\b_2}(r)|^2|1-\tilde{\b_3}(r)^2|\\
\leq&96\epsilon^2.
\end{align*}
Thus
\begin{align*}
\sum_{r\in\RR_1\cap\RR_2\cap\RR_3}|\tilde{\a_1}(r)\tilde{\a_2}(r)\tilde{\a_3}(r)(1-\tilde{\b_1}(r)^2
\tilde{\b_2}(r)^2\tilde{\b_3}(r)^2)|\leq&768\epsilon^2\min_{1\leq
i\leq
3}|\RR_i|\\
\leq&256\epsilon^2(12C_3^{12/5}\delta^{-12/5}+60C_4\delta^{-4}).
\end{align*}

On the other hand, by the H\"older inequality,
\begin{align*}
&\sum_{r\not\in\RR_1\cap\RR_2\cap\RR_3}|\tilde{\a_1}(r)\tilde{\a_2}(r)\tilde{\a_3}(r)(1-\tilde{\b_1}(r)^2
\tilde{\b_2}(r)^2\tilde{\b_3}(r)^2)|\\
\leq&2\sum_{r\not\in\RR_1\cap\RR_2\cap\RR_3}|\tilde{\a_1}(r)\tilde{\a_2}(r)\tilde{\a_3}(r)|\\
\leq&2\max_{r\not\in\RR_1\cap\RR_2\cap\RR_3}|\tilde{\a_1}(r)\tilde{\a_2}(r)\tilde{\a_3}(r)|^{1/13}
\sum_{r\in\Z_N}|\tilde{\a_1}(r)\tilde{\a_2}(r)\tilde{\a_3}(r)|^{12/13}\\
\leq&4\delta^{1/13}
\bigg(\sum_{r\in\Z_N}|\tilde{\a_1}(r)|^{12/5}\bigg)^{5/13}
\bigg(\sum_{r\in\Z_N}|\tilde{\a_2}(r)|^{12/5}\bigg)^{5/13}\bigg(\sum_{r\in\Z_N}|\tilde{\a_3}(r)|^4\bigg)^{3/13}\\
\leq&4\delta^{1/13}(6C_3^{12/5})^{10/13}(60C_4)^{3/13}.
\end{align*}

\end{proof}

\begin{Lem}
\label{pollard} Suppose that $0<\theta_1,\theta_2,\theta_3\leqslant
1$ with $\theta_1+\theta_2+\theta_3>1$. Let
$$
\theta=\min\{\theta_1,\theta_2,\theta_3,(\theta_1+\theta_2+\theta_3-1)/4\}.
$$
Suppose that $N$ is a prime greater than $2\theta^{-2}$, and
$X_1,X_2,X_3$ are subsets of $\Z_N$ with $|X_i|\geqslant\theta_i N$.
Then for any $y\in\Z_N$, we have
$$
|\{(x_1,x_2,x_3):\, x_i\in X_i,\ x_1+x_2+x_3=y\}|\geq\theta^3 N^2.
$$
\end{Lem}
\begin{proof}
See \cite[Lemma 3.3]{LiPan}.
\end{proof}

Now we are ready to prove Theorem \ref{chengoldbach}.
\begin{proof}[Proof of Theorem \ref{chengoldbach}]

For $1\leq i\leq 3$, let
$$
A_i'=\{x\in\Z_N:\, \a_i'(x)\geq\varpi/N\}.
$$
In view of (\ref{a3sum}) and Lemma \ref{a3upper},
$$
\frac{1+2\kappa}{N}|A_3'|+\frac{\varpi}{N}(N-|A_3'|)\geq\sum_{x\in\Z_N}\a_3'(x)=
\sum_{x\in\Z_N}\a_3(x)\geq 1-\kappa^2.
$$
Therefore $|A_3'|\geq (1-3\varpi)N$. Similarly, by (\ref{a12sum})
and Lemma \ref{a12upper}, for $i=1,2$
$$
\frac{1+2\kappa}{N}|A_i'|+\frac{\varpi}{N}(N-|A_i'|)\geq
\sum_{x\in\Z_N}\a_i(x)\geq6\varpi.
$$
Hence $|A_1'|,|A_2'|\geq \frac{9}{2}\varpi N$. Since
$|A_1'|+|A_2'|+|A_3'|\geq (1+6\varpi)N$, with the help of Lemma
\ref{pollard},
$$
\sum_{\substack{x_1\in A_1',\ x_2\in A_2',\ x_3\in A_3'\\
x_1+x_2+x_3=n'}}1\geq2\varpi^3N^2.
$$

By Lemma \ref{kappa}, we may choose $0<\delta,\epsilon\leq 1$ with
$\epsilon^{6C_3^{12/5}\delta^{-12/5}+60C_4\delta^{-4}}\geq\kappa$
satisfying
$$
3072\epsilon^2(C_3^{12/5}\delta^{-12/5}+5C_4\delta^{-4})+
72C_3^{24/13}C_4^{3/13}\delta^{1/13}\leq\varpi^6.
$$
Notice that $w=w(n)$ tends to infinity with $n$. So we may assume
that $w\geq\max\{20\kappa^{-2}+2,C_5^2\kappa^{-4}\}$. Since
$\max_{1\leq i\leq 3}|\RR_i|\leq
6C_3^{12/5}\delta^{-12/5}+6C_4\delta^{-4}$, the requirements of
Lemmas \ref{a3upper}, \ref{a12upper} and \ref{threesum} are
obviously satisfied.

Applying Lemma \ref{threesum},
\begin{align*}
&\sum_{\substack{x_1,x_2,x_3\in\Z_N\\
x_1+x_2+x_3=n'}}\a_1(x_1)\a_2(x_2)\a_3(x_3)\\\geq&\sum_{\substack{x_1,x_2,x_3\in\Z_N\\
x_1+x_2+x_3=n'}}\a_1'(x_1)\a_2'(x_2)\a_3'(x_3)-\frac{3072\epsilon^2(C_3^{12/5}\delta^{-12/5}+5C_4\delta^{-4})+
72C_3^{24/13}C_4^{3/13}\delta^{1/13}}{N}\\
\geq&\sum_{\substack{x_1\in A_1',\ x_2\in A_2',\ x_3\in A_3'\\
x_1+x_2+x_3=n'}}\bigg(\frac{\varpi}{N}\bigg)^3-\frac{\varpi^6}{N}\\
\geq&\frac{\varpi^6}{N}.
\end{align*}
This completes the proof.
\end{proof}

\end{document}